\documentclass[11pt]{article}
\usepackage{amsthm}
\usepackage{amssymb}
\usepackage{epsfig}
\usepackage{amsthm}
\usepackage{srcltx}
\usepackage[dvipsnames,usenames]{color}
\usepackage{paralist}
\usepackage{amsmath}
\usepackage{amsfonts}
\usepackage{float}

\newcounter{contador}
\newcounter{teoA}

\newtheorem{teo}[contador]{Theorem}
\newtheorem{lem}[contador]{Lemma}

\newtheorem{corol}[contador]{Corollary}
\newtheorem{nota}[contador]{Remark}

\newcounter{ex}

\newcommand{\dr}{{\rm{{d}}}}

\newcommand{\R}{{\mathbb R}}

\title{Phase portraits of random planar homogeneous vector fields\footnote{The authors are supported by
Ministry of Science and Innovation--State Research Agency of the
Spanish Government through grants MTM2016-77278-P (MICINN/AEI/FEDER,
UE, first and second authors) and DPI2016-77407-P
 (MICINN/AEI/FEDER, UE, third author). The first and second authors are also supported by the grant 2017-SGR-1617  from
AGAUR,  Generalitat de Catalunya. The third author acknowledges the
group's research recognition 2017-SGR-388 from AGAUR, Generalitat de
Catalunya.}}

\author{Anna Cima$^{(1)}$, Armengol Gasull$^{(1)}$ and V\'{\i}ctor Ma\~{n}osa$^{(2)}$
  \\*[.1truecm]
{\small \textsl{$^{(1)}$ Departament de Matem\`{a}tiques, Facultat
de Ci\`{e}ncies,}}
\\*[-.25truecm] {\small \textsl{Universitat Aut\`{o}noma de Barcelona,}}
\\*[-.25truecm] {\small \textsl{08193 Bellaterra, Barcelona, Spain}}
\\*[-.25truecm] {\small \textsl{cima@mat.uab.cat,
gasull@mat.uab.cat}}\\
\\*[-.25truecm] {\small \textsl{$^{(2)}$ Departament de Matem\`{a}tiques,}}
\\*[-.25truecm] {\small \textsl{Universitat Polit\`{e}cnica de Catalunya}}
\\*[-.25truecm] {\small \textsl{Colom 11, 08222 Terrassa, Spain}}
\\*[-.25truecm] {\small \textsl{victor.manosa@upc.edu}}}
\begin{document}

\maketitle
\begin{abstract}  In this paper, we study the probability of
occurrence of phase portraits in the set of random planar
homogeneous polynomial vector fields, of degree $n$. In particular,
for $n=1,2,3,$ we give the complete solution of the problem; that
is, we either give the exact value of each probability  of
occurrence or we estimate it by using the Monte Carlo method.
Remarkably is that all but two of these phase portraits are
characterized by the index at the origin and by the number of
invariant straight lines through this point.
\end{abstract}




\section{Introduction and main results}

Systems of ordinary differential equations, or equivalently vector
fields, are ubiquitous tools in the mathematical modelling of
physical phenomenon. When studying parametric families of such
vector fields typically one tries to know the different types of
phase portraits that can appear in the family and what are their
characterizations: that is, to obtain the \emph{parametric
bifurcation diagram} of the family. A second level of study could be
to establish the measure of the occurrence, or repetitiveness of
each possible phase portrait among the different phase portraits in
the family. Looking for this coincidence, or other dynamic
characteristic such as the existence of attractors, invariant
straight lines,  or limit cycles, are the aims of some recent works:
\cite{AL,CGM,LPP} and \cite{PLPP}.

The study of the probability of appearance  of the different phase
portraits may be also of practical interest. For instance,  in some
cases, to model real-life situations, we deal with dynamical systems
having  natural variability in the parameters. To include this
uncertainty in the model, we must treat with families of vector
fields whose parameters vary randomly. See \cite{CS10,GS,SC09} for
example.

Historically, the interest in this approach can be also traced back  to
A.N. Kolmogorov, from whom there is an interesting anecdote that we
have encountered in \cite{LBA}. According to V.I.~Arnold
\cite{Arnold}, it seems that Kolmogorov proposed to his students of
the Moscow State University a problem that should give a statistical
measure of the quadratic vector fields having limit cycles. He gave
them several hundreds of such vector fields with randomly chosen
coefficients. Each student was asked to find the number of limit
cycles of his/her vector field. Surprisingly, the experiment gave
that none of the vector fields had any limit cycle, although, since
hyperbolic limit cycles form open sets in the space of coefficients,
the probability to get them for a random choice of the coefficients
is positive. In fact, the probability of having obtained this result is very low: if $p$ is the probability of a quadratic vector field to have some limit cycle then, assuming independence, the probability that none of the vector fields had limit cycles is $(1-p)^n$, where $n$ is the number of students involved.  In \cite{AL}, $p$ is estimated to be $p\approx 0.0323$, hence assuming $n=200$ we get
$(1-p)^{200}\approx 0.0014$.

In this paper we  study the probabilities  of the different phase
portraits appearing in random planar homogeneous polynomial vector
fields of degree $n.$ Their phase portraits in the Poincar\'{e} disk
(\cite{DLA}) are well understood and they can be described in terms
of the number of real zeroes of some associated polynomials, their
relative position and several algebraic inequalities involving these
zeroes and the parameters of the vector fields, see~\cite{A,CL,D}.

 We will focus on the
quadratic and cubic cases and for completeness we also will tackle
the  linear case.  In fact, the results for  $n=1$ are well-known
(\cite{LPP,PLPP}) and, for instance,  they are proposed as an
exercise  in the S.~Strogatz book \cite[Exer. 5.2.14, p.
143]{Strog}. We include a simple and self-contained proof for this
case. Our characterization of the different phase portraits, for
$n\in\{1,2,3\},$ will be mainly based  on the computation of the
index at the origin, $i,$ and on the study of the number of
invariant lines through it, $l.$ As we will see, these two numbers
provide a complete algebraic characterization, by means of
inequalities,  of all the different classes of phase portraits with
positive probability when $n$ is 1 or $2,$ and almost a complete one
for $n=3.$

In the cases that we have not been able to compute analytically the
probabilities,  we give their estimations  by using the Monte Carlo
method (\cite{BFS,Morgan}) and the algebraic classification tools
developed in this paper.

We will say that
\begin{align}
    F_n(x,y)&=P_n(x,y)\frac{\partial}{\partial x}+Q_n(x,y)
    \frac{\partial}{\partial y}\nonumber\\{}&=\left(\sum\limits_{i+j=n}
     A_{i,j} x^i y^j\right)\frac{\partial}{\partial x}+\left(\sum\limits_{i+j=n} B_{i,j} x^i y^j\right)\frac{\partial}{\partial y}\label{e:rhvf}
\end{align} is a \emph{random planar homogeneous polynomial vector field of degree $n$} if  all the variables
$A_{i,j}$ and $B_{i,j}$ are independent normal random variables with
zero mean and standard deviation one. For short, we will write
$A_{i,j}\sim \mathrm{N}(0,1)$ and $B_{i,j}\sim \mathrm{N}(0,1).$ The
hypothesis that all  the coefficients are $\mathrm{N}(0,1)$  is
commonly used,  see for instance~\cite{AL,CGM,LPP,PLPP} and it is
quite natural and well motivated. In
Section~\ref{s:probabilityspace} we briefly recall this motivation.

In general we will denote by $P(W)$ the probability that a phase portrait of type $W$ occurs, modulus time orientation.
Next results  collect  our main achievements.

\begin{teo}\label{t:teolineals}
    Consider the linear random vector fields $F_1$ of the form~\eqref{e:rhvf}.
    Their phase portraits (modulus time orientation) with positive probability
     are the three ones shown in Figure~\ref{f:dibuix1}. They are completely determined
      by the couple  $(i,l)=$(index, number of invariant lines) at the origin.  Moreover,
    \[
    P(L_1)=\frac12,\quad P(L_2)= \frac{\sqrt{2}}{2}-\frac{1}{2}\simeq 0.20711 \quad\mbox{and}\quad P(L_3)= 1-\frac{\sqrt{2}}{2}\simeq 0.29289.
    \]
    Finally, the probability of the origin to be a global attractor
    is $1/4$ and the one of being a global repeller is also $1/4.$
\end{teo}

\begin{figure}[H]
    \begin{center}
        \includegraphics[scale=1]{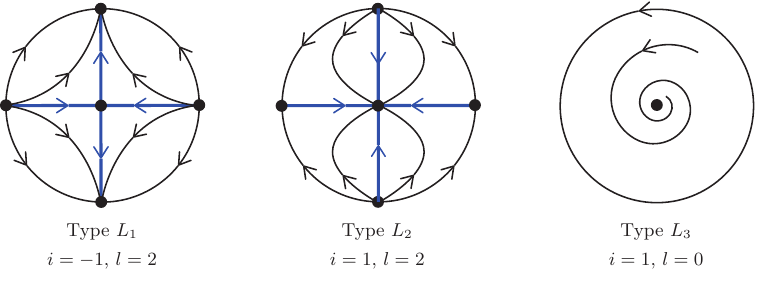}
        \caption{Phase portraits of planar linear vector fields with positive
            probability,  modulus time orientation.}
        \label{f:dibuix1}
    \end{center}
\end{figure}

To state our results for $n=2,3$ we need to introduce next two values
\begin{align}
\Lambda_2=&\frac{1}{\pi}\,\int_{-\infty }^{\infty }\!{\frac {\sqrt
        {2\,{t}^{8}+8\,{t}^{6}+13\,{t
            }^{4}+8\,{t}^{2}+2}}
        {  \left( {t}^{2}+1
            \right)  \left( t^4+t^2+1 \right)              }}\,{\rm
    d}t\simeq 1.64343,\label{eq:l2}\\
\Lambda_3=&\frac{1}{\pi}\,\int_{-\infty }^{\infty }{\frac{\sqrt
        {2{t}^{8}+4\,{t}^{6}+ 12\,{t}^{4}+4\,{t}^{2}+2}}{ \left( {t}^{2}+1
            \right)  \left( t^4+1 \right)
}}\,{\rm d}t\simeq 1.81225,\label{eq:l3}
\end{align}
which, as we will see in item $(c)$ of Theorem~\ref{te:tot},
correspond to the expected number of invariant straight lines for
vector field~\eqref{e:rhvf} through the origin, for $n=2$ and $n=3,$
respectively. These values are obtained by using the powerful tools
introduced in the nice  paper of Edelman and~Kostlan~\cite{EK}.

\begin{teo}\label{t:teoquad}
      Consider the quadratic random vector fields $F_2$ of the form~\eqref{e:rhvf}.
    Their phase portraits with positive probability are the five ones shown in Figure~\ref{f:dibuix2}.
    They are completely determined by the couple  $(i,l)=$(index, number of invariant lines) at the origin.
     Moreover, their corresponding probabilities, in addition to $\sum_{j=1}^5 P(Q_j)=1,$  satisfy
    \[
P(Q_1)+P(Q_2)+P(Q_3)=\frac{\Lambda_2-1}2\simeq 0.32172,\quad P(Q_1)=P(Q_3)+P(Q_5),
    \]
and they are estimated in Table~\ref{t:taula2}.
\end{teo}

\begin{figure}[H]
    \begin{center}
        \includegraphics[scale=1]{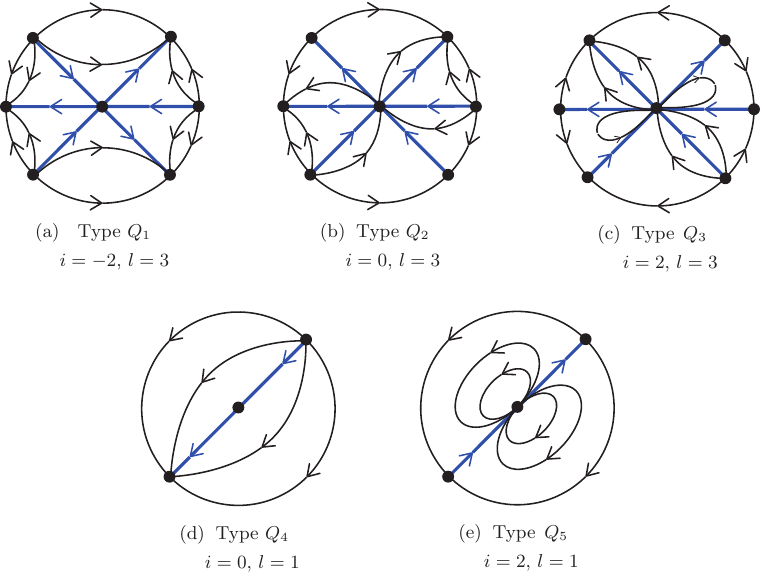}
        \caption{Phase portraits of planar quadratic homogeneous vector fields  with positive probability.}
        \label{f:dibuix2}
    \end{center}
\end{figure}

\begin{table}[htbp]
    \begin{center}
        \begin{tabular}{|l|c|}
            \hline
            \phantom & Estimated probabilities \\
            \hline $P(Q_1)$ & $0.11588$ \\ \hline$P(Q_2)$ & $0.18583$ \\
            \hline $P(Q_3)$ & $0.01999$ \\\hline $P(Q_4)$ &
            $0.58242$\\\hline $P(Q_5)$ & $0.09588$ \\\hline
        \end{tabular}
        \qquad
        \begin{tabular}{|l|c|}
            \hline
            \phantom & Observed frequency \\
            \hline $i(f)=-2$ & $0.11588$ \\
            \hline $i(f)=0$ & $0.76825$ \\
            \hline $i(f)=2$ & $0.11587$ \\ \hline
        \end{tabular}
    \end{center}
    \caption{Estimations of the probabilities $P(Q_j)$, $j=1,2,\ldots,5,$ of quadratic random vector fields, $F_2,$ of the
    form~\eqref{e:rhvf}, in accordance with Figure~\ref{f:dibuix2}.}
    \label{t:taula2}
\end{table}

To estimate the values $P(Q_j)$ we have applied the Monte Carlo
method to $10^8$   quadratic homogeneous vector fields, randomly
generated; that is, we have set up $10^8$ vectors in $\mathbb{R}^6$
whose entries are pseudo-random numbers simulating the six
independent random variables of the quadratic random vector field
$F_2,$ with $\mathrm{N}(0,1)$ distribution. This algorithm consists
in a repeated random sampling and gives estimations of the desired
probabilities due to the laws of large numbers and that  of the
iterated logarithm, see \cite{BFS,Morgan}. In almost all cases, the
corresponding phase portrait for the generated sample is obtained by
checking the sign of algebraic inequalities among the obtained
values of these six independent random variables. These signs allow
us to know the index of the origin and number of invariant straight
lines through it.

 It turns out that using $m$
samples, the Monte Carlo method gives the sought value with an
absolute error of order $O\big(((\log\log m)/m)^{1/2}\big),$ which
practically behaves as $O(m^{-1/2}).$ Since in our simulations we
take $m=10^8,$ the approaches found for the desired probabilities
will have an absolute error of order $10^{-4}.$ The results are
given in Table~\ref{t:taula2}. In the right-hand side of that table
the observed frequencies are also collected in terms of the index of
the critical point.

Notice that the  estimated probabilities of the phase portraits
$Q_1,Q_2,Q_3,Q_4$ and $Q_5$  are in good agreement
with the  relations given in Theorem~\ref{t:teoquad}.

Observe that precisely $P(Q_1)+P(Q_2)+P(Q_3)$ is the probability of
having three invariant straight lines and
$P(Q_4)+P(Q_5)=\frac{3-\Lambda_2}2\approx 0.67828$ is  the
probability of having one invariant straight line. These results are
also in good agreement with the ones of \cite{AL}. More in detail,
in that paper the authors give estimations of the probabilities of
the different phase portraits of structurally stable quadratic
vector fields, not necessarily homogeneous. It is easy to see that
structurally stable vector fields are  vector fields with full
probability. Their phase portraits  near infinity are related with
the ones  of homogeneous quadratic vector fields. In particular, the
number of critical points at the equator of the Poincar\'{e} disk is
twice  the number of invariant straight lines  of the corresponding
homogeneous quadratic vector field. In \cite{AL}, attending to the
behavior at infinity, ten different families are considered; the
first five having only a couple of diametrically opposed
singularities at infinite and the other ones having three pairs of
singularities. If, for the first five families in Table 2 in
\cite{AL}, we add their observed frequencies, also obtained by Monte
Carlo simulation, we get $0.25422+0.25424+0.07396+0.07749+0.01839=
0.6783.$ This value is in good agreement with the probability of
having one invariant line given above.

For $n=3,$ the results are similar to the quadratic case.

\begin{teo}\label{t:teocub}
      Consider cubic random vector fields $F_3$ of the form~\eqref{e:rhvf}.
   Their  phase portraits (modulus time orientation) with positive probability are the nine ones shown in
      Figure~\ref{f:dibuix3}. Except for the phase portraits $C_3$ and $C_4$, they are  determined
       by the couple  $(i,l)=$(index, number of invariant lines) at the origin.
    Moreover, their corresponding probabilities, in addition to $\sum_{j=1}^9 P(C_j)=1,$  satisfy
    \begin{align*}
&4\sum_{j=1}^5 P(C_j)+2\sum_{j=6}^8 P(C_j)=\Lambda_3,\quad
P(C_1)=P(C_5)+P(C_8),\\ &P(C_2)+P(C_6)=P(C_3)+P(C_4)+P(C_7)+P(C_9),
    \end{align*}
    and they are estimated in Table~\ref{t:taula3}.
\end{teo}

\begin{table}[H]
    \begin{center}
        \begin{tabular}{|l|c|}
            \hline
            \phantom & Estimated probabilities  \\
            \hline $P(C_1)$ & $0.00909$ \\
            \hline $P(C_2)$ & $0.04193$ \\
            \hline  $P(C_3)$ & $0.00615$ \\
            \hline  $P(C_4)$ &
            $0.02394$\\
            \hline  $P(C_5)$ & $0.00065$ \\
            \hline  $P(C_6)$ & $0.44897$\\
            \hline $P(C_7)$ & $0.28521$\\\hline
        $P(C_8)$ & $0.00845$\\\hline  $P(C_9)$ & $0.17561$\\
            \hline
        \end{tabular}\qquad
        \begin{tabular}{|c|c|}
            \hline
            \phantom & Observed frequencies \\
            \hline $i(f)=-3$ & $0.00909$ \\ \hline $i(f)=-1$ & $0.49090$ \\
            \hline $i(f)=1$ & $0.49091$ \\ \hline $i(f)=3$ &
            $0.00910$ \\ \hline\hline ${\bf 0}$ is a global attractor & $0.24238$\\
             \hline ${\bf 0}$ is a global repeller & $0.24238$\\
            \hline
        \end{tabular}
    \end{center}
\caption{Estimations of the probabilities $P(C_j)$,
$j=1,2,\ldots,9$,  of cubic random vector fields, $F_3,$ of the
    form~\eqref{e:rhvf}, in accordance with Figure~\ref{f:dibuix2}.} \label{t:taula3}
\end{table}

Again,  the estimated probabilities  given in Table~\ref{t:taula3} are
in good agreement with the relations given in
Theorem~\ref{t:teocub}.

\begin{nota}\label{re:temps}
    Notice that when  $n$ is even, all phase portraits of homogeneous vector fields are
     conjugated with the ones obtained reversing the orientation of all trajectories.
      This is so, because the associated differential equations are invariant with the
       change $(x,y,t)\rightarrow (-x,-y,-t).$ This is no more true when $n$ is odd. For instance,
    if we consider phase portraits $L_2$
    or $L_3$  in Figure~\ref{f:dibuix1}, or $C_2$ or $C_5$ in Figure~\ref{f:dibuix3}, and we reverse the arrows, the new
    phase portraits are not conjugated, but  topologically equivalent.
\end{nota}

\begin{figure}[H]
    \begin{center}
        \includegraphics[scale=1]{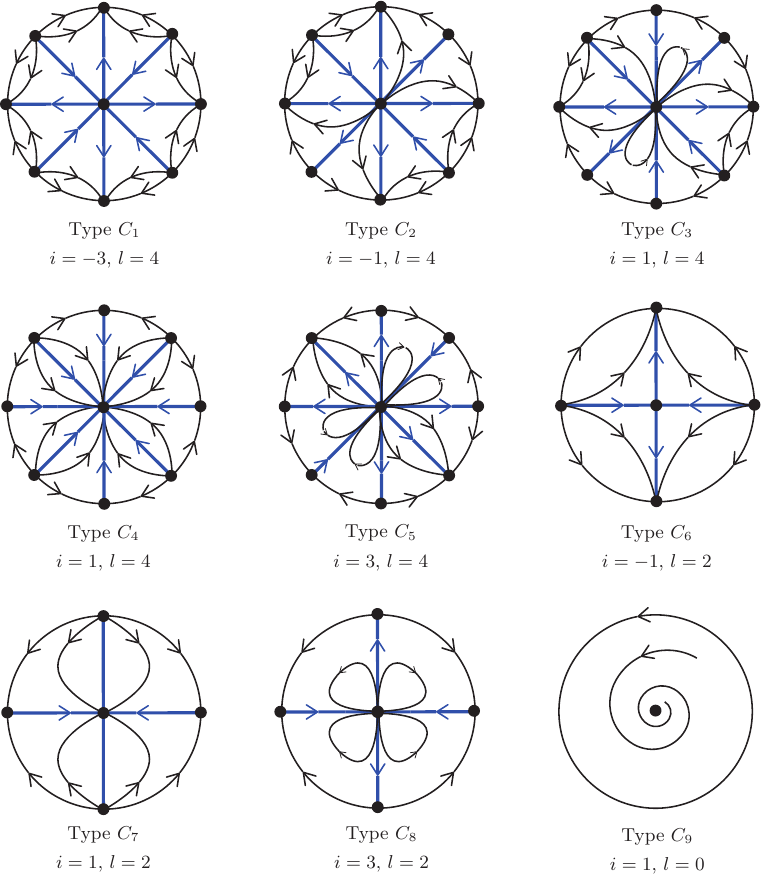}
        \caption{Phase portraits of planar cubic homogeneous vector fields  with positive probability,   modulus time orientation.}
        \label{f:dibuix3}
    \end{center}
\end{figure}

 Since the distribution of the coefficients of random planar homogeneous
  vector fields of degree $n$  is absolutely  continuous and with a positive density,
 the phase portraits that have positive probability coincide with the ones
  that are structurally stable, in the world of polynomial homogeneous
 vector fields od degree $n,$ see for instance~\cite{LPR} for the precise
  definitions. Fixed $n,$ the number of topologically equivalent different
 classes of phase portraits $S_n$ of such systems is given in Table~1  of that paper. The first ten values are reproduced in
 Table~\ref{t:1}. Observe that the value $S_2$ coincides with the number of different
  phase portraits given in Theorem~\ref{t:teoquad}.
 On the other hand, the values $S_1$ and $S_3$ are bigger than the corresponding ones of
   Theorems~\ref{t:teolineals} and~\ref{t:teocub}, respectively.
 This is so, because recall that by Remark~\ref{re:temps}, when~$n$ is even the phase portraits are
  invariant with respect the change $(x,y,t)\rightarrow(-x,-y,-t)$
 and when $n$ is odd this is no more true. In fact, changing  the direction of all
  the arrows, when $n=1,$ only one of the three obtained vector fields
 is topologically equivalent to itself and so, the three phase portraits of
  Theorem~\ref{t:teolineals} split  into $1+2\times2=5=S_1$ phase portraits.
 Similarly, for $n=3,$ only four of the  phase portraits of Theorem~\ref{t:teocub}
  remain invariant and as a consequence  we obtain $4+2\times5=14=S_3$ phase portraits.
 Finally,  notice also that by the invariance of the distribution of
the coefficients of vector field~\eqref{e:rhvf}, the probability of
a phase portrait and the one  obtained by changing the arrows
direction coincide. For instance, in the linear case, as a corollary
of Theorem~\ref{t:teolineals}, the probabilities of attracting node
is $(\sqrt{2}-1)/4$ and the one of repelling node is the same.

\begin{table}[h]
    \begin{center}
        \begin{tabular}{|c||c|c|c|c|c|c|c|c|c|c|}
            \hline
            $n$&1&2&3&4&5&6&7&8&9 &10 \\
            \hline $S_n$  &5&5&14&13&34&31&85&77&221&203 \\
            \hline
        \end{tabular}
    \end{center}
    \caption{Number of phase portraits with positive probability $S_n$ for vector field~\eqref{e:rhvf}.}\label{t:1}
\end{table}

We organize the paper as follows. In Section~\ref{s:homogeneous} we recall the main
 results of qualitative theory that allow to study the phase portraits of homogeneous
  vector fields together with the usual methods to know the multiplicity, the
   index $i$ and the number of invariant straight of planar vector fields, $l.$
    Section~\ref{s:probabilityspace}  collects all the probabilistic preliminaries and
    results that we will use. In particular we discuss the reason that justifies our
     definition of random planar  homogeneous vector field \eqref{e:rhvf}, as well as
      a tool  to study the expected number of real zeroes of random polynomials (\cite{EK}).
       Sections~\ref{ss:planarlinearsystems}, \ref{s:quad} and \ref{s:cubics} are devoted to
        prove Theorems~\ref{t:teolineals},~\ref{t:teoquad} and~\ref{t:teocub}, respectively.
        Finally in the Appendix we study the number and location of the real roots of polynomials
        of degree 3 and 4. As we will see, these results will be needed to study the couple $i$ and $l,$
         that we recall  gives the index of the origin and the number of invariant lines of the vector field through this point.

\section{Algebraic tools for the classification of the phase portraits}\label{s:homogeneous}

Consider planar polynomial vector fields of type
\begin{equation}\label{hom}
 f(x,y)=p_n(x,y)\frac{\partial}{\partial x}+q_n(x,y)
\frac{\partial}{\partial y},
\end{equation}
where $p_n(x,y)$ and $q_n(x,y)$ are real homogeneous polynomials of
degree $n.$ Their phase portraits on the Poincar\'{e} disk can be
obtained following the procedure detailed in~\cite{A}. For the cases
$n=2,3$ they are given in \cite{A,D}  and \cite{CL} respectively.
For $n=1,$ they correspond to the well-known linear homogeneous
vector fields.
 Since our objective
is to compute the probability of the different available phase
portraits  we need to characterize them in terms of algebraic
equalities and inequalities. As we will see, except in two cases that
will be explained in Section \ref{s:cubics}, the characterization of
their phase portraits occurring with positive probability can be
done by using only two objects: the index $i$ of the vector field
associated to \eqref{hom} at the origin and the number of invariant
straight lines $l,$ through this point.

We dedicate next two subsections first to recall how to know  the
number of invariant  straight lines  through the origin and secondly
to explain how to compute the index of an isolated critical point.
Since there are no essential differences, most results in this
second section deal with  $m$-dimensional vector fields.

\subsection{Invariant straight lines}\label{ss:invlines}

 It is
straightforward to see that the line $\alpha x+\beta y=0$ is
invariant by the flow of the homogeneous system~\eqref{hom}  if and
only if $\alpha x+\beta y$ is a factor of $x\,q_n(x,y)-y\,p_n(x,y).$
Due to the homogeneity we see that the slopes of these invariant
straight lines, different from $x=0$, are the values of
$\kappa\in\R$ that satisfy
\begin{equation}\label{e:Fdirecions}
t_n(\kappa):=q_n(1,\kappa)-\kappa\,p_n(1,\kappa)=0.
\end{equation} Hence, if $t_n(\kappa)\not\equiv0$ and
$p_n(0,y)\not\equiv 0$   the number of invariant straight lines $l\le
n+1$ is exactly the number of real zeros of $t_n(\kappa).$ If $t_n(\kappa)\equiv0$ or
$p_n(0,y)\equiv 0,$  the number of invariant straight lines is either
infinity or $l+1.$

Finally, we remark that the number of real roots of a general
polynomial of a fixed degree can be  characterized in terms of
algebraic inequalities among its coefficients. This fact can be
proved for instance by using Sturm sequences. See the Appendix for
the explicit results for polynomials of degree $3$ and $4.$

\subsection{Index at isolated singular points}\label{ss:index}

To simplify the notation,  we will do the following abuse of
notation: we  will write  $f=(f_1\ldots,f_m)$  to denote an analytic
map of $\mathbb{R}^m$ but also a finite map germ, or even  the
vector field $f=\sum_{j=1}^m f_j\partial/\partial x_j$. In all the
cases, the meaning will be explicitly stated or clearly deducible
from the context.  In this section we will assume that $f(0)=0,$
since we will always work with the singular point at the origin.

In general, if $f:(\R^m,0)\rightarrow (\R^m,0)$ is a
continuous map and $0$ is isolated in $f^{-1}(0),$ then \emph{the
index of $f$ at $0$}, $\mathrm{ind}(f),$ is defined as the degree of
the map  $f/||f||:{\mathbb
S}_{\epsilon}\rightarrow {\mathbb S}_{1},$ where ${\mathbb
S}_{\epsilon}$ is the boundary of a ball of radius $\epsilon$,  ${\mathbb B}_{\epsilon}$,
such that $f^{-1}(0)\bigcap {\mathbb B}_{\epsilon}=\{0\}.$ If $f$ is
differentiable, this number can be computed as the sum of the signs
of the Jacobian of $f$ at all the  preimages near $0$ of a regular
value of $f$ near $0,$  see \cite[Lemmas 3 and 4]{Milnor}.

If $f$ is a $C^\infty$ map we can consider the local
ring of germs of $C^{\infty}$ functions at the origin
$C_{\bf{0}}^{\infty}(\R^m)$ and the quotient ring
$$Q(f)=C_{\bf{0}}^{\infty}(\R^m)/(f),$$
where  $(f)$ denotes the ideal generated by the components of $f.$
It holds that when $\mathbf{0}$ is an isolated singularity then
$Q(f)$ is a finite dimensional real vector space and its dimension
is called the multiplicity of $f$ at 0, $\mu(f).$  In fact, when $f$
is polynomial this multiplicity  coincide with the number of complex
preimages of any regular value near 0.

As mentioned above, we want to determine the index by means of
 algebraic inequalities among the coefficients of our vector
fields. This can be done, for instance, by using the Eisenbud-Levine
signature formula for the index, see \cite{EL} or also \cite[Chap.
5]{AVG} for instance.

\begin{teo}[\cite{EL}]\label{te:EL}
Let $f:(\R^m,0)\rightarrow (\R^m,0)$ be a $C^\infty$ a finite map
germ, and let $\bar{J}\in Q(f)$ be the residue class of the Jacobian
of $f,$ $J=\det(\mathrm{D}f)$. If $\varphi:Q(f)\rightarrow \R$ is a
linear functional such that $\varphi(\bar{J})>0,$ and if $< ,
>\,=\,< ,
>_{\varphi}$ is the symmetric bilinear form on the ring $Q(f)$
defined by
$$<p ,q >=\varphi (pq) \mbox{ for } p,q\in Q(f),$$
then the index of $f$ at ${\bf{0}}$, is
$\mathrm{ind}(f)=\,\mathrm{signature}< , >.$
\end{teo}

Since the signature of a quadratic form is the difference between
the number of positive eigenvalues and the number of  negative
 ones
of its associated matrix, we need to know  this number in terms of
algebraic inequalities among the coefficients of its characteristic
polynomial. This always can be done by using Sturm sequences
(\cite{SB}). Since we will study the index of the quadratic and
cubic planar homogeneous vector fields, we need to know the number
of real roots for polynomials of degree $3$ and $4$. For polynomials
of degree 3 we also need to know the number of positive and negative
real roots.  These characterizations are done in the Appendix.

We will also use the following simple properties of the index.

\begin{lem}\label{l:lem1}
Let $f:(\R^m,0)\rightarrow (\R^m,0),$  be a $C^\infty$ finite map
germ. Then:
\begin{enumerate}[(a)]
\item  If  $g=(f_1,f_2,\ldots, -f_m),$ then $\mathrm{ind}(f)=-\mathrm{ind}(g).$
\item  If  $g=(-f_1,-f_2,\ldots, -f_m)$ then $\mathrm{ind}(f)=(-1)^m\,\mathrm{ind}(g).$
\end{enumerate}
\end{lem}

\begin{proof} $(a)$ The proof is straightforward by using the algebraic formula for the
index given in Theorem \ref{te:EL}. Clearly $Q(f)=Q(g)$ and the
Jacobian  determinant of $g,\, J_g$, is minus the Jacobian
determinant of $f, \,J_f.$ Now consider $\varphi_f:Q(f)\rightarrow
\R$ a linear functional such that $\varphi_f(\bar{J_f})>0,$ and let
$A$ be the associated matrix. Then, if we define
$\varphi_g(p)=-\varphi_f(p),$ $\varphi_g$ is a linear functional
which satisfies that $\varphi_g(\bar{J_g})>0,$ and its corresponding
matrix $B$ is $B=-A.$ This fact implies that $P_B(\lambda)=0$ if and
only if $P_A(-\lambda)=0,$ where $P_A$ and $P_B$ are the
characteristic polynomials of $A$ and $B$ respectively. Hence the
signature of $<
>_{\varphi_g}$ is minus the signature of $<
>_{\varphi_f}.$

$(b)$ This proof follows similarly.
\end{proof}

Finally, it is also known,  see \cite{CGT}, that
$$|\mathrm{ind}(f)|\le (\mu(f))^{1-\frac1m}\, \mbox{ and }\,\mathrm{ind}(f)\equiv
\mu(f)\,\,(\text{mod}\,\, 2).$$

\medskip

Applying the above results to \eqref{hom} we have:

\begin{corol}\label{co:homo} Let  $f$ be  a homogeneous planar vector field
\eqref{hom} of degree $n.$ Assume that the origin is an isolated
singularity. Then $\mu(f)=n^2$ and
\[
|\mathrm{ind}(f)|\le n \, \mbox{ with }\,\mathrm{ind}(f)\equiv
n\,\,(\mathrm{mod}\,\, 2).
\]
In particular, for $n=1$ the index $i\in\{-1,1\},$ for $n=2$ the
index $i\in\{-2,0,2\}$ and for $n=3$ the index $i\in\{-3,-1,1,3\}.$
\end{corol}

For $m=2,$ there is an alternative way to compute the index of an
isolated singularity of an analytic planar vector field $f.$ The
Bendixon's index formula says that $\mathrm{ind}(f)=1+(e-h)/2,$
where $e$ and $h$  are respectively the number of elliptic and
hyperbolic sectors of this singularity.

\section{Probabilistic tools for the study of the phase portraits}\label{s:probabilityspace}

In this section we start motivating our definition of random homogeneous vector field \eqref{e:rhvf}.
We also prove Theorem~\ref{te:tot} that contains relationships among several probabilities of the phase
portraits of~\eqref{e:rhvf} with positive probability.


Consider families of vector fields in the plane that are written as
\begin{equation}\label{e:firsteq}
Y=A_1\,Y_1+A_2\,Y_2+\cdots +A_k\,Y_k,
\end{equation}
where $Y_1,Y_2,\ldots,Y_k$ are fixed vector fields on $\R^2$ and
$A_1,A_2,\ldots,A_k$ are random variables to be fixed, taking values on $\R.$ In
this way each event $\omega$ consists on a given vector field
\begin{equation}\label{e:firsteq2}
Y=a_1\,Y_1+a_2\,Y_2+\cdots +a_k\,Y_k,
\end{equation}
with $a_j=A_j(\omega)$ (notice that in the whole paper we will use capital letters to denote the
coefficients of vector fields or polynomials when they are random variables, and lowercase letters
when they are real numbers). In order to give a measure of the different
phase portraits a fundamental issue is to determine which is the
natural election of distribution of the random variables $A_j$. Only
after this step is properly  done we can ask for the probabilities of some
dynamical features.

It is natural to assume that the random variables $A_j$ are
\emph{continuous}, \emph{independent} and \emph{identically
distributed} on $\R$.  The \emph{principle of indifference}
\cite{Conrad} would seem to indicate that we should take a
distribution for the variables $A_j$ in such a way that the vector
$(A_1,\ldots,A_k)$ had some kind of uniform distribution in
$\mathbb{R}^k$. But there is no uniform distribution for unbounded
probability spaces. However notice that any phase portrait is
invariant under positive linear time scalings. Hence all systems
\eqref{e:firsteq2} with parameters $(\lambda a_1,\ldots,\lambda
a_k)$ with $\lambda>0$ have topologically equivalent phase
portraits. Thus it is also natural to consider a  distribution  for
the coefficients $A_j$ such that the random vector $
\left({A_1}/{S},{A_2}/{S},\ldots,{A_k}/{S}\right),$ where
$S=\sqrt{\sum_{j=1}^kA_j^2},$ has a uniform distribution on the sphere
${\mathbb S}^{k-1} \subset \R^k$. Next result will justify our
choice of distribution for the random variables $A_j.$

\begin{teo} (\cite{CGM, Mars,Mull})
    Let $X_1,X_2,\ldots,X_k$ be independent and identically distributed one-dimensional
    random variables with a continuous positive density function $f$. The random
    vector
    $
    X=\left({X_1}/{S},{X_2}/{S},\ldots,{X_k}/{S}\right),
    $
    where $S=\sqrt{\sum_{j=1}^{k} X_j^2}$, has a uniform distribution in
    $\mathbb{S}^{k-1}\subset\mathbb{R}^k$ if and only if all $X_j$ are normal
    random variables with zero mean.
\end{teo}

Hence, in the random vector field \eqref{e:firsteq} we consider the
probability space $(\Omega,\mathcal{F},P)$ where $\Omega=\R^k$,
$\mathcal{F}$ is the $\sigma$-algebra generated by the open sets of
$\R^k$ and $P:\mathcal{F}\rightarrow [0,1]$ is the probability with
joint density
\begin{equation}\label{E:densityf}
\psi(a_1,a_2,\ldots,a_k)\,=\,\frac{1}{\left(\sqrt{2\pi}\right)^k}\,\mathrm{e}^{-\frac{a_1^2+a_2^2+\ldots
+a_k^2}{2}}.\end{equation} Notice that for simplicity it is not
restrictive to consider that the variance for the centered normal
random variables is 1.

The above explanation justifies the definition of random homogeneous
vector field given in \eqref{e:rhvf}. Notice that they are under the
formulation of \eqref{e:firsteq} with $k=2n+2$ and each $Y_i$ either as $x^j y^{n-j}\frac{\partial}{\partial x}$ or
$x^m y^{n-m}\frac{\partial}{\partial y},$ for some $j$ and $m$ with $0\le j\le n,$ and $0\le m\le n.$

\begin{nota}\label{p:measurable} Observe that the probability density function associated
to \eqref{e:rhvf} is positive. Therefore,  any non-empty event
described by algebraic inequalities is measurable and has positive
probability. Similarly, the measurable events such  that in their
description appears a non trivial  algebraic equality have
probability zero.
\end{nota}

Next result collects most of the information needed to prove our
main results.

\begin{teo}\label{te:tot}
\begin{enumerate}[(a)]
\item Set $u_n(k)=P(\mbox{\rm the index of } \eqref{e:rhvf}\, \mbox{\rm  at }  {\bf{0}}\, \mbox{\rm  is } k).$ Then:
\begin{enumerate}[(i)]
\item  $u_n(k)\ne0$ if and only if  $|k|\le n$  and
$k\equiv n \, (\mathrm{mod}\, 2),$
\item $u_n(k)=u_n(-k)$ and as a consequence the expected index of the random vector field~\eqref{e:rhvf} is
$\sum_{|k|\le n} k u_n(k)=0.$
\end{enumerate}
\item Set $a_n=P( \bf{0}\,\,  \mbox{\rm is a global attractor for } \eqref{e:rhvf} )$ and $r_n=
P( \bf{0} \, \,\mbox{\rm is a global repeller for }
\eqref{e:rhvf}).$ Then $a_n=r_n.$ Moreover $a_n\ne0$ if and only if
$n$ is odd.

\item Set $\,\ell_n(k)=P(\mbox{\rm vector field } \eqref{e:rhvf}\,\, \mbox{\rm has exactly }  k\,\, \mbox{\rm invariant straight lines}).$
Then $\ell_n(k)\ne0$ if and only if  $k\le n+1$  and
$k\equiv n+1 \, (\mathrm{mod}\, 2).$ Moreover the expected number of
invariant straight lines is $\Lambda_n= \sum_{k=0}^{n+1} k
\ell_n(k).$ In particular,  $\Lambda_1=\sqrt{2}$
      and $\Lambda_2$ and $\Lambda_3$ are given in   \eqref{eq:l2} and \eqref{eq:l3}, respectively.
\end{enumerate}
\end{teo}

\begin{proof}  First observe that  from the results of \cite{A} and
    \cite{EL}  the events we are interested in are measurable, because
    they are defined by algebraic inequalities (see next sections to
    have more details of how to characterize these events).
Observe that by Remark \ref{p:measurable} the probability of
$P_n(x,y)$ and $Q_n(x,y)$ to have a common factor is zero since this
fact is characterized by an algebraic equality among the
coefficients of $P_n$ and $Q_n$ (as can be seen taking successive
resultants).  Hence we can always assume that the origin is
isolated. In particular, we are under the hypotheses of Corollary
\ref{co:homo}, and therefore the statement $(i)$ of item $(a)$
follows. To prove item $(ii)$ consider the associated  random vector
field $G_n(x,y)=P_n(x,y)\frac{\partial}{\partial x}-Q_n(x,y)
    \frac{\partial}{\partial y}$. Observe that
    $P(\mathrm{ind}(F_n)=k)\,=\,P(\mathrm{ind}(G_n)=k)$ because, due to their
    symmetry,  the variables $B_{i,j}\sim-B_{i,j}\sim \mathrm{N}(0,1)$.  From
    Lemma \ref{l:lem1} $(a)$ we have $\mathrm{ind}(G_n)=-\mathrm{ind}(F_n),$ and the
    result follows.

$(b)$ Recall that $\bf{x}=\bf{0}$ is a global attractor (resp. a
global repeller) if  $\lim_{t\to\infty}(x(t),y(t))=\bf{0}$ for all
the solutions $(x(t),y(t))$ of the system (resp.
$\lim_{t\to-\infty}(x(t),y(t))=\bf{0}$). In the homogeneous case, if
$\bf{x}=\bf{0}$ is a global attractor  it cannot have neither
elliptic nor  hyperbolic sectors since otherwise there would appear
an invariant straight line with an escaping orbit. Hence, according
to the Bendixon's index formula, $\mathrm{ind}(f)=1+(e-h)/2,$ since
$e=h=0$ we have $\mathrm{ind}(f)=1$ and therefore, from
Corollary~\ref{co:homo},   $n$ must odd. The same argument holds for
the existence of global repellers. Using again the tools introduced
in \cite{A} to determine the phase portraits we know that the events
of having a global attractor or a global repeller are measurable.
Moreover, in particular  we know that $a_n=r_n=0$ when $n$ is even.

Notice that $f$ has $\bf{x}=\bf{0}$ as a global attractor if and only
    if $-f$ has $\bf{x}=\bf{0}$ as a global repeller, and again   due to
     the symmetry of  the distribution of the random variables $A_{i,j}$ and $B_{i,j},$ we know that $f$ and $-f$ have the same distribution.
     Hence $a_n=r_n$ as we wanted to prove. Finally, it is not
     difficult to find an open set of vector fields with the desired
     properties, so $a_n=r_n\ne0$ when $n$ is odd.

$(c)$ In the paper~\cite{EK}, Edelman and~Kostlan give the expected
number of real zeros of any equation of type
$$A_0f_0(t)+A_1f_1(t)+\ldots +A_kf_k(t)=0,$$
where $A_j$ $i=0,\ldots,k$ are normal distributed with zero mean, not
 necessarily  being neither identically distributed nor independent, and  $f_j$ are
  differentiable functions. Following the notation in \cite[Thm 3.1]{EK} if we consider a
random polynomial
    \begin{equation}\label{e:Prandom}
    P(\kappa)=\sum_{j=0}^{n+1} P_j \kappa^j
    \end{equation} where $P_j$ are normal random variables
     with mean zero and covariance matrix $M_{n+1}$;
     $w(\kappa)=M_{n+1}^{1/2}\cdot (1,\kappa,\kappa^2,\ldots,\kappa^{n+1})^t$
     and $\mathbf{w}(\kappa)=w(\kappa)/||w(\kappa)||$, then the expected
     number of real zeros of $P$ is given by the Edelman-Kostlan formula:
    $$
    \int_{-\infty}^{\infty} \frac{1}{\pi} ||\mathbf{w}'(\kappa)||\,\dr \kappa.
    $$
    A straightforward computation gives that the random
     polynomials $T_n(\kappa)$ in \eqref{e:Fdirecions} that control
      the number of invariant straight lines of the random system
      \eqref{e:rhvf}
are
\[
T_n(\kappa)=B_{n,0}+\sum_{j=1}^{n}\big(B_{n-j,j}-A_{n-j+1.j-1}\big)\kappa^j+A_{0,n}\kappa^{n+1}.
\]
Moreover since all $A_{i,j}$ and $B_{i,j}$ have $\mathrm{N}(0,1)$
distribution and are independent, it holds that
\[
T_n(\kappa)=C_0+\sum_{j=1}^n C_j \kappa^j+C_{n+1} \kappa^{n+1},
\]
where $C_0$ and $C_{n+1}$ are  $\mathrm{N}(0,1)$ and the other $C_j$
are centered normal random variables with standard deviation
$\sqrt{2}$, being all the $n+2$ random variables independent. Thus
we deal with an expression of the form \eqref{e:Prandom}. In
particular, their covariance matrices are
    $$
    M_2=\left(\begin{array}{ccc}
    1 & 0 &0 \\
    0 & 2 & 0\\
    0 & 0 & 1
    \end{array}\right),\quad M_3=\left(\begin{array}{cccc}
    1 & 0 &0 &0\\
    0 & 2 & 0& 0\\
    0 & 0 & 2& 0\\
    0& 0 & 0 & 1
    \end{array}\right)\quad \mbox{and}\quad M_4=\left(\begin{array}{ccccc}
    1 & 0 &0 &0&0\\
    0 & 2 & 0& 0&0\\
    0 & 0 & 2& 0&0\\
    0& 0 & 0 & 2& 0\\
    0 & 0 & 0& 0 & 1
    \end{array}\right),
    $$
    respectively. The values of $\Lambda_n$ can be obtained by  straightforward
    computations. For instance, for $n=1,$
    \[
{\bf w}(\kappa)=\left(\frac{1}{1+\kappa^2},\frac{\sqrt
2\kappa}{1+\kappa^2},
\frac{\kappa^2}{1+\kappa^2}\right)\quad\mbox{and}\quad {\bf
w}'(\kappa)=\left(\frac{-2\kappa}{(1+\kappa^2)^2},\frac{\sqrt
2(1-\kappa^2)}{(1+\kappa^2)^2},
\frac{2\kappa}{(1+\kappa^2)^2}\right).
    \]
Hence,
    \[
        \Lambda_1= \frac{1}{\pi}\,\int_{-\infty}^{\infty} ||\mathbf{w}'(\kappa) ||\,\dr \kappa=
           \frac{1}{\pi}\,\int_{-\infty}^{\infty} \frac{\sqrt{2}}{(1+\kappa^2)}\,\dr \kappa=\sqrt{2}.
    \]
    For $n=2,3$ we skip the details.
    ~\end{proof}

\begin{nota} By using the same tools that in the proof  of item~$(c)$ of Theorem~\ref{te:tot},
other values of $\Lambda_n$ can be obtained. For instance
\[
\Lambda_4\approx 1.94648,\quad \Lambda_5\approx 2.05788,\quad
\Lambda_6\approx 2.15303,\quad \Lambda_7\approx
2.23603,\ldots\quad \Lambda_{10}\approx  2.43552.
\]
\end{nota}

\section{Random linear vector fields and proof of  Theorem~\ref{t:teolineals}}\label{ss:planarlinearsystems}

We give a simple self contained proof of
Theorem~\ref{t:teolineals} and also an alternative proof based on the results
given in item $(c)$ of Theorem~\ref{te:tot}, as a corollary of the
computation of the expected number of invariant straight lines,
$\Lambda_1.$

\begin{proof}[Proof of Theorem~\ref{t:teolineals}]  For shortness we denote $l_j=P(L_j),$ $j=1,2,3.$
The only phase portraits of  linear vector fields given by inequalities
among the parameters  of the vector field are the saddle; the (generic)
node with two different eigenvalues; and the focus. Their phase
portraits in the Poincar\'{e} disk  (\cite{DLA}), modulus time orientation, are the ones of
Figure~\ref{f:dibuix1}. Moreover  $l_1+l_2+l_3=1.$ A distinction
between nodes
    and focus comes from the fact that for the case of focus there are
    not invariant straight lines while generic nodes have two invariant
    straight lines.

    It is also well-known that the index of the origin in the saddle case is $-1$ while
     in the node and focus cases it is $+1.$ Therefore, by item $(a)$ of Theorem~\ref{te:tot},
      $l_1=l_2+l_3$. By using both equalities we get that $l_1=\frac12$ and $l_2+l_3=\frac12.$
      Hence to prove the theorem it suffices to calculate either $l_2$ or $l_3.$

    Let us prove that $l_3=1-\frac{\sqrt{2}}{2}.$ The characteristic polynomial
     associated to the linear vector fields $(Ax+By)\frac{\partial}{\partial x}+(Cx+Dy)\frac{\partial}{\partial y}$ is
    $\lambda^2-(A+D)\lambda+AD-BC.$ Hence, the probability that
    $\mathbf{x}=\mathbf{0}$ is a focus is $l_3=P((A-D)^2+4BC<0).$ Since $A-D$ is a
    normal random variable with zero mean and standard deviation $\sqrt{2}$, to compute
    the probability of  $\mathbf{x}=\mathbf{0}$ being a focus we consider the
     random vector $(X,Y,Z):=(B,C,A-D)$, where $X,Y$ and $Z$ are independent normal variables
      with zero mean, and $\sigma_X=\sigma_Y=1$ and $\sigma_Z=\sqrt{2}$. Thus, the joint density function  is:
    $$
    \psi(x,y,z)=\frac{1}{4\,\pi^{3/2}}\,\mathrm{e}^{-\frac{2x^2+2y^2+z^2}4}.
    $$
    Hence
    $
    l_3=P(Z^2+4XY<0)=\int_{K} \psi(x,y,z) \dr x\,\dr y\,\dr z$, where
    $K:=\{(x,y,z)\in\mathbb{R}^3:  z^2+4xy<0\}$. To compute this integral, we perform
    the change of variables
    $$
    x:=r\sin \left( t \right) \cos \left( s \right),
    y: =r\sin \left( t \right) \sin \left( s \right),
    z: =\sqrt {2}r\cos \left( t \right),
    $$
    where $t\in \left(-\frac{\pi}{2},\frac{\pi}{2}\right),$ $s\in
    (0,2\pi)$ and $r>0.$ Since the determinant of the Jacobian of the change is
    $\sqrt{2} \sin(t) r^2$, we have
    $$l_3=\frac{1}{4 \pi^{3/2}}\int_{\widetilde{K}}\int_0^{+\infty}\sqrt{2} \sin(t)r^2{{\rm
            e}^{-1/2 {r}^{2}}}\dr r \,\dr s\, \dr t
    $$
    where
    $\widetilde{K}=\{(s,t):\cos(t)^2+2 \sin(t)^2 \sin(s) \cos(s)<0\}.$
    Recall that $$\int_0^{\infty}
    \sqrt{2} r^2{{\rm e}^{-1/2 {r}^{2}}} \dr r = \sqrt{\pi}.$$ To
    calculate the remainder part of the integral,
    $\int_{\widetilde{K}}\sin(t)\dr t \dr s $, we consider the curve
    $\cos^2(t)+2 \sin^2(t) \sin(s) \cos(s)=0,$ that is,
    $t=\pm\arctan\frac{1}{\sqrt{-\sin(2s)}}$, which is depicted in Figure~\ref{f:dibuixa}.
    \begin{figure}[htb]
        \begin{center}
            \includegraphics[scale=0.35]{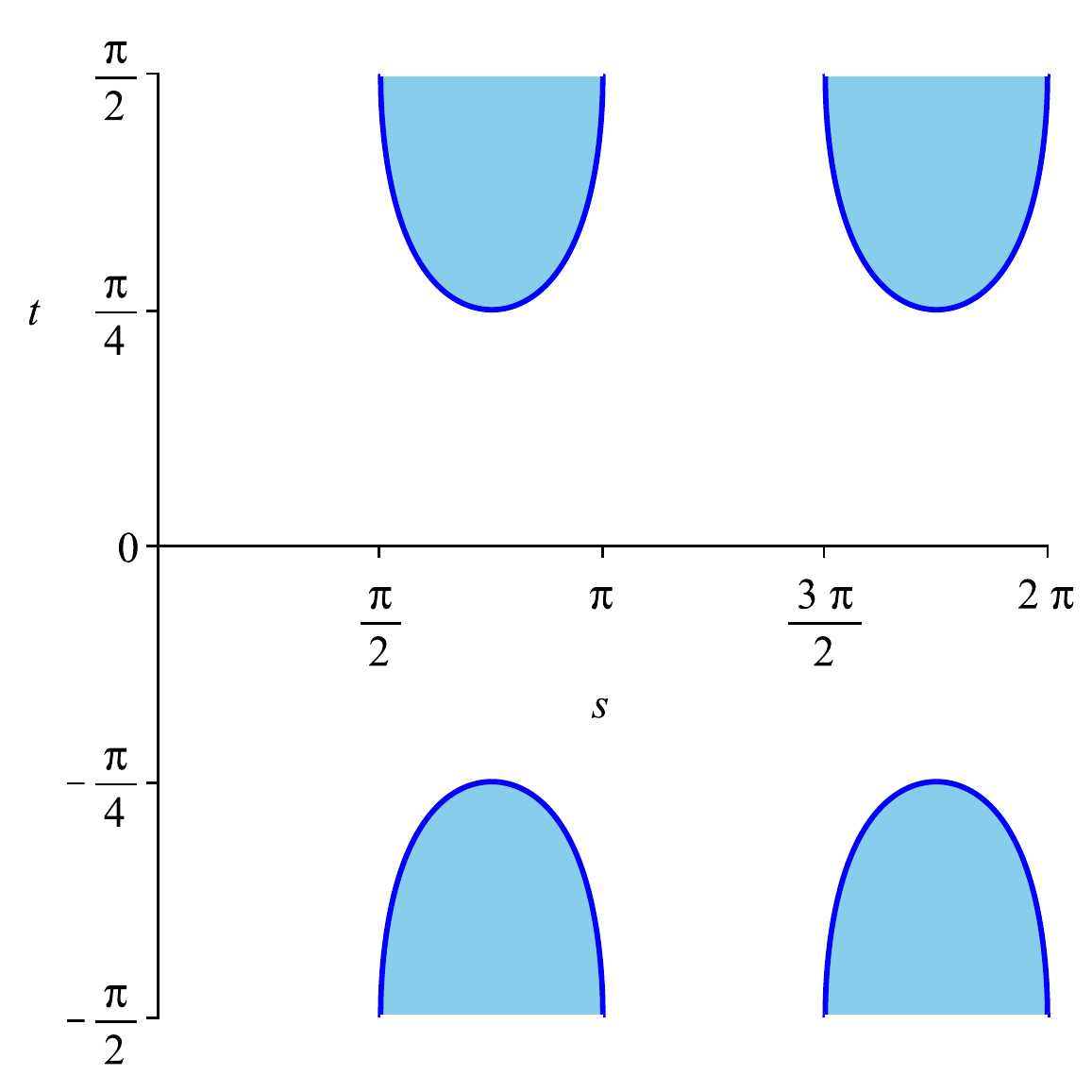}
            \caption{Graphic of the curve $\cos^2(t)+2 \sin^2(t) \sin(s) \cos(s)=0$. The set $\widetilde{K}$ is shadowed.}
            \label{f:dibuixa}
        \end{center}
    \end{figure}

    It is easy to see that  the shadowed area in Figure~\ref{f:dibuixa} corresponds to $\widetilde{K}.$ Using the symmetries of this curve, and
    calling $g_1(s):=\arctan\frac{1}{\sqrt{-\sin(2s)}}$, we get that
    $$\int_{\widetilde{K}}\sin(t)\,\dr s\,\dr t=4\,\int_{\pi/2}^{\pi}\int_{g_1(s)}^{\pi/2}\,\sin(t)\,\dr t\,\dr s.$$
    Taking into account that $\cos(\arctan(\xi))=1/\sqrt{\xi^2+1}$, and 
    doing the change $u=2s$, we get
    $$\int_{\pi/2}^{\pi}\int_{g_1(s)}^{\pi/2}\,\sin(t)\,\dr t\,\dr s=    
\frac{1}{2} \, \int_{\pi}^{2\pi}  \sqrt{ \frac{\sin(u)}{\sin(u)-1}}\,\dr u=
    \frac{\pi}{2}\,(2-\sqrt{2}).$$
    Hence
    \[
    l_3=\frac{\sqrt{\pi}}{4 \pi^{3/2}}\int_{\widetilde{K}}\sin(t)\,\dr t\,\dr s=\frac{1}{4 \pi}\frac{4\pi}2(2-\sqrt 2)=1-\frac {\sqrt 2}2
    \]
    and the computation of all probabilities  follows.

Finally,  notice that if a vector field has the phase portrait $L_2$
(or $L_3$) and we perform a time reversion $t\to -t$ then
we get the same picture with reversed time arrows, and hence the two
phase portraits are not conjugated, but topologically equivalent. In fact, the pictures of $L_2$
and $L_3$ in Figure \ref{f:dibuix1} correspond to global attractors.
For the phase portrait $L_1$ a time reversal gives
 a  conjugated one.

 Due to the symmetry of the random variables that give the coefficients of the random homogeneous vector field \eqref{e:rhvf} we
 know that $f$ and $-f$ have the same distribution. Hence the
probability to get an attractive node (resp. focus) coincides with
the probability to get a repulsive node (resp. focus), see also item (b) of Theorem~\ref{te:tot}. Therefore,
\begin{multline*}
P({\bf 0}\, \text{ is and attractor})=P({\bf 0}\, \text{ is an
attractive node})+P({\bf 0}\, \text{in an attractive
focus})\\=\frac{1}{2}\,P({\bf 0}\, \text{ is a
node})+\frac{1}{2}\,P({\bf 0}\, \text{ is a
focus})=\frac{1}{2}\,l_2+\frac{1}{2}\,l_3=\frac{1}{2}(l_2+l_3)=\frac{1}{4}.
\end{multline*}
~\end{proof}

\bigskip

\begin{proof}[Alternative proof of Theorem~\ref{t:teolineals}] Set $l_j=P(L_j), j=1,2,3.$
Following the same steps that  in the previous proof Theorem~\ref{t:teolineals}, we obtain that
$l_1=\frac12$ and $l_2+l_3=\frac12.$ The only difference in this
proof  is how to find $l_2$ and $l_3.$ By item $(c)$ of Theorem
\ref{te:tot},
\[
2\big(l_1+l_2\big)+0\cdot l_3=2\ell_1(2)+0\cdot
\ell_1(0)=\Lambda_1=\sqrt{2}.
\]
By joining the three equalities we obtain that $l_3=1-\sqrt{2}/2$
and the result follows.~\end{proof}

\section{Random homogeneous quadratic vector fields}\label{s:quad}

Since we are interested in knowing the phase portraits
\eqref{e:rhvf} with positive probability, we can assume that the
components of our vector fields have not common factors (see Section
\ref{ss:index}). Hence  the multiplicity  at $\mathbf{0}$ is
$n^2$. For the quadratic case ($n=2$), it is 4.

To prove Theorem~\ref{t:teoquad} we need to algebraically
characterize the indices at the origin and also the number of
invariant straight lines though it for quadratic homogenous vector
fields
\begin{equation}\label{quad}
f_2(x,y)=(ax^2+bxy+cy^2)\, \frac{\partial}{\partial x}+
\,(dx^2+exy+fy^2) \frac{\partial}{\partial y},
\end{equation}
with the origin being an isolated singularity. This is so because,
as we will see, the phase portrait of the quadratic homogeneous
vector fields with positive measure in the event space (which is
given by the parameter space $\Omega=\mathbb{R}^6$) is characterized
by these two numbers.

\subsection{Index at the origin}

By Corollary \ref{co:homo} the index at the origin of \eqref{quad}
is $-2,0$ or $2.$ The next result, explicitly characterizes
$\mathrm{ind}(f_2)$ for generic vector fields, in terms of algebraic
inequalities (hence open sets in the parameter space $\mathbb{R}^6$).

Set
$$\lambda:=-\frac{ae-bd}{af-cd},\quad\mu:=-\frac{bf-ce}{af-cd},\quad j:=4(af-cd)(1-\lambda\mu).$$
We introduce an additional genericity condition for vector field
\eqref{quad}. We will say that the vector field $f_2$ is \emph{well-posed}
if $af-cd\neq 0,$ $\lambda\cdot \mu\neq 0,$ $\lambda\mu-1\neq
0$ and $\lambda+\mu\neq 0$. Observe that well-posed random
homogeneous quadratic  vector fields \eqref{quad} have full
probability.

\begin{teo}\label{t:teo15} Let $f_2$ as in \eqref{quad} be a well-posed vector field, and let $\epsilon=\pm 1$ be such that
$\epsilon\,j>0.$ Then the following holds:
\begin{enumerate}
\item[(a)] $\mathrm{ind}(f_2)=0$ if and only if
$\lambda\mu-1<0.$
\item[(b)] $\mathrm{ind}(f_2)=2$ if and only if
$\lambda\mu-1>0$ and $\epsilon(\lambda+\mu)>0.$
\item[(c)] $\mathrm{ind}(f_2)=-2$ if and only if
$\lambda\mu-1>0$ and $\epsilon(\lambda+\mu)<0.$
\end{enumerate}
\end{teo}

\begin{proof} We will start proving that
$$Q(f_2)=C_{\bf{0}}^{\infty}(\R^2)/(f_2)\,=\,\left\langle 1,\bar{x},\bar{y},\bar{x}\bar{y} \right\rangle.$$
Notice that the dimension of this space is 4 as expected because we
already knew that  the multiplicity of $f_2$ at ${\bf{0}}$ is 4.
Indeed, observe that since the components of $f_2$ are  zero in
$Q(f_2),$ and using that $f_2$ is well-posed, we get that
$$a(d\bar{x}^2+e\bar{x}\bar{y}+f\bar{y}^2)-d(a\bar{x}^2+b\bar{x}\bar{y}+c\bar{y}^2)=
(ae-bd)\bar{x}\bar{y}+(af-cd)\bar{y}^2=0,$$ hence
$\bar{y}^2=\lambda\bar{x}\bar{y}.$ Similarly
$\bar{x}^2=\mu\bar{x}\bar{y}.$ It implies that
$$\bar{x}^2\bar{y}=(\mu\bar{x}\bar{y})\bar{y}=\mu\bar{x}\bar{y}^2=\mu\bar{x}(\lambda\bar{x}\bar{y})=\mu\lambda\bar{x}^2\bar{y}.$$
Since $\mu\lambda\ne 1$ we get that $\bar{x}^2\bar{y}=0$ in $Q(f_2).$
It easily implies that all the monomials of degree $k\ge 3$ are zero
in the quotient ring. Furthermore, some computations give that
\[
 j=2(bf-ce)\lambda+4(af-cd)+2(ae-bd)\mu=4(af-cd)(1-\lambda\mu).
\]
Hence, the residual class of the Jacobian of $f_2$ is
$\bar{J}=j\bar{x}\bar{y}.$

Let $\varphi:Q(f_2)\rightarrow \R$ be the functional sending
$\bar{x}\bar{y}$ to $\epsilon=\pm 1$ with $\epsilon j>0$, and sending
the other basis elements to $0.$ Then the matrix of $< ,
>_{\varphi}$ with respect this basis is:
$$\left( \begin {array}{cccc} 0&0&0&\epsilon\\ \noalign{\medskip}0&\mu
\,\epsilon&\epsilon&0\\
\noalign{\medskip}0&\epsilon&\epsilon\,\lambda &0\\
\noalign{\medskip}\epsilon&0&0&0\end {array} \right).$$ The
characteristic polynomial is given by
$P(z)=(z^2-1)\left(z^2-(\lambda+\mu)\epsilon z+\lambda
\mu-1\right).$ Let $z_1,z_2$ be the two roots of
$z^2-(\lambda+\mu)\epsilon z+\lambda \mu-1=0.$ Since $z_1
z_2=\lambda \mu-1$, $z_1+z_2=(\lambda+\mu)\epsilon $ and the
signature of a quadratic form is the difference between positive
eigenvalues and the negative ones, the result follows.
\end{proof}

\subsection{Number of invariant straight lines trough the origin}\label{ss:il}

Concerning the number of invariant straight lines passing through
the origin of \eqref{quad}, as we have proved in Section
\ref{ss:invlines}, generically it suffices to look at equation \eqref{e:Fdirecions}, which writes
$t_2(\kappa)=-c\kappa^3+(f-b)\kappa^2+(e-a)\kappa+d=0,$ with $c\ne0.$ Again,
generically a cubic equation  has either three different real
roots or one simple real root. These two possibilities are
distinguished by the discriminant of $t_2$, $\Delta_{t_2}$  which is
given by
$$\Delta_{t_2}=\frac{1}{c^4}\left[-27\,{c}^{2}{d}^{2}-18\,cd \left( b-f \right)  \left( a-e \right) +4
\,d \left( b-f \right) ^{3}+ \left( b-f \right) ^{2} \left( a-e
 \right) ^{2}-4\,c \left( a-e \right) ^{3}\right].$$ The result is
 that $t_2$ has three different real roots if and only
 $\Delta_{t_2}>0$ and $t_2$ has just one simple real root if and only
 $\Delta_{t_2}<0.$ These two cases give a full probability event for
  the random vector field \eqref{e:rhvf} with $n=2.$

\subsection{Proof of Theorem~\ref{t:teoquad} and Table~\ref{t:taula2}}

\begin{proof}[Proof of Theorem~\ref{t:teoquad}]
 The phase portarits of quadratic homogeneous vector fields are well-know,   see~\cite{A,D}.
 It is not difficult to see that the only one with positive probability  are the five ones
 given in Figure~\ref{f:dibuix2}. This is so, because all the other ones are characterized
 by some equality among the coefficients and the vector field, and so, they have probability 0
 of appearance. Set $q_j=P(Q_j)$, $j=1,2,\ldots,5.$ By looking at these phase portraits and using
 the notation of Theorem~\ref{te:tot} we have that
 \[
u_2(-2)= q_1,\, u_2(0)= q_2+q_4 ,\, u_2(2)= q_3+q_5,\, \ell_2(1)=q_4+q_5 ,\, \ell_2(3)=q_1+q_2+q_3.
 \]
By Theorem~\ref{te:tot} we know that  $u_2(2)=u_2(-2),$  and
$\ell_2(1)+3\ell_2(3)=\Lambda_2.$ Hence
$q_1=q_3+q_5$ and $q_4+q_5+3(q_1+q_2+q_3)=\Lambda_2.$ Using that $\sum_{j=1}^5 q_j=1,$ these two
equalities give the two ones in the statement of the theorem.
\end{proof}

Taking into account the three relations among the $q_j$ given in
Theorem~\ref{t:teoquad}, only two more relations among these five
probabilities have to be found in order obtain their exact values.
For instance one of these new relations could be $u_2(0)= q_2+q_4.$
 By using item (a) of Theorem \ref{t:teo15} and the probability
density function given in \eqref{E:densityf} we obtain
\begin{align*}
u_2(0)
&=\int_{K} \mathrm{e}^{-\frac{a^2+b^2+c^2+d^2+\mathrm{e}^2+f^2}{2}}\dr a\,\dr b\,\dr c\,\dr d\,\dr e\,\dr f,
\end{align*} where $K=\{(a,b,c,d,e,f)\in\R^6:(ae-bd)(bf-ce)<(af-cd)^2\}.$
We have not been able  to calculate the above integral analytically. It could be approximated by several
numerical methods. In fact, one of the most used is  Monte Carlo method to evaluate multiple integrals.
For this reason we have decided to compute directly the values $P(Q_j)$ by direct Monte Carlo simulation
instead of approaching $u_2(0).$

The results of Theorem~\ref{t:teo15} and the ones of
Section~\ref{ss:il}  that give algebraic inequalities among the
coefficients to know the index at the origin, and the number of
invariant straight lines through it, respectively, allow to
determine the phase portrait of any well-posed homogeneous quadratic
vector field with an isolated singularity. Notice that these vector
fields have full probability. We use   these results to know which
phase portrait corresponds to each sample generated by Monte Carlo
method.   The obtained approximations of the values $q_j$ are given
in Table~\ref{t:taula2}.

\begin{nota}\label{re:invariants} Another way to distinguish the five phase portraits with positive probability for planar quadratic homogeneous systems  consists on using the \emph{invariants approach} adopted in \cite{ALSV} for quadratic systems. There, the nature and configuration of the singular points at the infinity in the Poincar\'{e} Compactification are given in terms of some invariants called $\eta,\mu_0$ and $\kappa$, which are certain homogeneous polynomials of degree 4 on the coefficients of the systems. The phase portraits  $Q_i$, $i=1,\ldots, 5$ correspond with the Configurations 1,5,7,30 and 34 respectively, given in Figure 6 and classified in Diagrams 1 and 2 of \cite{ALSV}. In fact, the invariant $\eta$ is the numerator of the discriminant $\Delta_{t_2}$ given in Section \ref{ss:il}.
\end{nota}

\section{Random homogeneous cubic vector fields}\label{s:cubics}

 This section mimics the previous one but with more involved computations.
 There is only one essential difference, as we will see,  for first time there appear
 two phase portraits that are neither conjugated nor equivalent but share index and
 number of invariant straight lines.

Consider a cubic homogeneous vector field
\begin{equation}\label{cubic} f_3(x,y)= (ax^3+bx^2y+cxy^2+dy^3)\,\frac{\partial}{\partial
x}+ (ex^3+fx^2y+gxy^2+hy^3)\,\frac{\partial}{\partial y}.
\end{equation}
Arguing as in Section~\ref{s:quad} we can assume that it has the origin as an isolated singularity and hence it has multiplicity $9.$

We start studying the index of \eqref{cubic}. By Corollary \ref{co:homo} we already know that
the only indices of the origin for \eqref{cubic} are $-3,-1,1,3.$ Set:
\begin{align*}
r:=&-\frac{bf-df}{ah-de}, \quad s:=-\frac{ch-dg}{ah-de}, \quad
p:=-\frac{af-be}{ah-de}, \quad q:=-\frac{ag-ec}{ah-de},\\
h_1:&=\frac{r(r+sq)+s(1-ps)}{1-ps}, \,\,  h_2:=\frac{r+sq}{1-ps}, \,\,
h_3:=\frac{pr+q}{1-ps}, \,\, h_4:=\frac{p(1-ps)+q(pr+q)}{1-ps},\\
j:=&\left( 3af-3be \right)h_1+
\left( 6ag-6ce \right)h_2+ \left( 9ah+3bg-3cf-9de \right)\\&+ \left( 6b
h-6df \right)h_3+ \left( 3ch-3dg \right)h_4.
\end{align*}
We also will need:
$$\begin{array}{lr}
\alpha:=-\epsilon\,(1+h_1+h_4),\\
\beta:=h_1h_4-h_2^2-h_3^2+h_1+h_4-1,\\
\gamma:=\epsilon\,(h_1h_3^2+h_2^2h_4-h_1h_4-2h_2h_3+1),
\end{array}$$
where for $j\ne0,$ $\epsilon \in\{-1,1\}$ is such that $\epsilon j>0,$ and
\[
C_2:=\alpha \beta-9\gamma,\quad
D_3:=-27\gamma^2+18\alpha\beta\gamma-4\alpha^3\gamma+\alpha^2\beta^2-4\beta^3.\]
We introduce a  genericity
condition for this cubic case. We say that $f_3$ is well-posed if $ah-ed\neq 0$ and
$ps\neq 1$; $r\cdot s\cdot p\cdot q\neq 0$ and $h_1\cdot h_2\cdot
h_3\cdot h_4\cdot j \neq 0$; $\alpha \cdot \beta \cdot \gamma\neq 0$ and
$C_2 \cdot D_3\neq 0.$ As in the quadratic case,  well-posed random cubic vector
fields have full probability.

\begin{teo}\label{indexcubics} Let $f_3$ as in \eqref{quad} be a well-posed cubic vector field. Then the following holds:
\begin{enumerate}
\item[(a)] $\mathrm{ind}(f_3)=-3$ if and only if
$D_3>0\,,\,C_2>0\,,\,\beta>0\,,\,\gamma>0.$
\item[(b)] $\mathrm{ind}(f_3)=-1$ if and only if either $$\begin{array}{lr}
D_3<0\,,\,\gamma>0 \,\,\text{or}\\
D_3>0\,,\,\gamma<0\,,\,C_2>0\,\,\text{or}\\
D_3>0\,,\,\gamma<0\,,\,C_2<0\,,\,\beta<0.
\end{array}$$
\item[(c)] $\mathrm{ind}(f_3)=1$ if and only if either $$\begin{array}{lr}
D_3<0\,,\,\gamma<0 \,\,\text{or}\\
D_3>0\,,\,\gamma>0\,,\,C_2<0\,\,\text{or}\\
D_3>0\,,\,\gamma>0\,,\,C_2>0\,,\,\beta<0.
\end{array}$$
\item[(d)] $\mathrm{ind}(f_3)=3$ if and only if
$D_3>0\,,\,C_2<0\,,\,\beta>0\,,\,\gamma<0.$
\end{enumerate}
\end{teo}

\begin{proof}
Since the components of $f_3$ are zero in $Q(f_3),$ we get that
$\bar{x}^3=r\bar{x}^2\bar{y}+s\bar{x}\bar{y}^2$ and $\bar{y}^3=p\bar{x}^2\bar{y}+q\bar{x}\bar{y}^2$
and simple computations give
$$\bar{x}^4=h_1\bar{x}^2\bar{y}^2\,\,,\,\,\bar{x}^3\bar{y}=h_2\bar{x}^2\bar{y}^2\,\,,\,\,
\bar{x}\bar{y}^3=h_3\bar{x}^2\bar{y}^2\,\,,\,\,\bar{y}^4=h_4\bar{x}^2\bar{y}^2.$$
From the above equalities and taking into account that $f_3$ is
well-posed we get that all the monomials of degree greater than four
are zero in $Q(f_3).$ Hence,
$$Q(f_3)=\left<1\,,\,\bar{x}\,,\,\bar{y}\,,\,\bar{x}^2\,,\,\bar{x}\bar{y}\,,\,\bar{y}^2\,,\,
\bar{x}^2\bar{y}\,,\,\bar{x}\bar{y}^2\,,\,\bar{x}^2\bar{y}^2\right>,$$
with a basis of nine elements, as expected.
Since the Jacobian of $f_3$ is
\begin{multline*}
\left( 3af-3be \right) {x}^{4}+ \left( 6ag-6ce \right) y{x}^{3}+
\left( 9ah+3bg-3cf-9de \right) {y}^{2}{x}^{2}\\+ \left( 6b h-6df
\right) {y}^{3}x+ \left( 3ch-3dg \right) {y}^{4}\end{multline*} its
residual class is $j\bar{x}^2\bar{y}^2,$ with $j\ne0.$

Let $\varphi:Q(f_3)\rightarrow \R$ be the functional sending
$\bar{x}^2\bar{y}^2$ to $\epsilon=\pm 1$ with $\epsilon j>0$ and
sending the other basis elements to $0.$ Then the matrix of $< ,
>_{\varphi}$ with respect this basis is:

$$\left( \begin {array}{ccccccccc} 0&0&0&0&0&0&0&0&\epsilon
\\ \noalign{\medskip}0&0&0&0&0&0&\epsilon\,{h_2}&\epsilon&0
\\ \noalign{\medskip}0&0&0&0&0&0&\epsilon&\epsilon\,{h_3}&0
\\ \noalign{\medskip}0&0&0&\epsilon\,{h_1}&\epsilon\,{h_2}&
\epsilon&0&0&0\\ \noalign{\medskip}0&0&0&\epsilon\,{h_2}&\epsilon& \epsilon\,{h_3}&0&0&0\\
\noalign{\medskip}0&0&0&\epsilon&\epsilon\,
{h_3}&\epsilon\,{h_4}&0&0&0\\
\noalign{\medskip}0&\epsilon\,{h_2}&\epsilon&0&0&0&0&0&0\\
\noalign{\medskip}0&\epsilon&\epsilon\, {h_3}&0&0&0&0&0&0\\
\noalign{\medskip}\epsilon&0&0&0&0&0&0&0&0
\end {array} \right).$$
The characteristic polynomial is given by
$$P(z)=-(\epsilon+z)\cdot(\epsilon-z)\cdot (z^2-\epsilon(h_2+h_3)z+h_2h_3-1)\cdot (z^2+\epsilon(h_2+h_3)z+h_2h_3-1)\cdot S(z)$$
where
$S(z)=z^3+\alpha\,z^2+\beta\,z+\gamma,$
and $\alpha,\beta, \gamma$ are defined above. In order to compute
the signature of this matrix it is not difficult to see that the only relevant part of
the characteristic polynomial is~$S.$ The conditions on the
coefficients of a cubic polynomial equation to give the number of
positive and negative zeros are established in Lemma
\ref{signaturacubics} in the Appendix. Applying it we get the
result.
\end{proof}

Concerning the number of invariant straight lines passing through
the origin for vector field \eqref{cubic}, generically we have to look at equation:
$t_3(\kappa)=-d{\kappa}^{4}+ \left( -c+h \right) {\kappa}^{3}+
\left( -b+g \right) {\kappa}^{2}+ \left( -a+f \right) \kappa+e,$ with $d\ne0.$
Again, under more generic assumptions, $t_3$ has either four different
real roots, or two simple  real roots, or  no real root.
These three possibilities can be distinguished by the sign of three algebraic expressions depending
on its coefficients,  given in
Lemma~\ref{l:app18}. Once more,  these
inequalities imply that the set of
parameters for which the vector field has a given number of invariant straight lines
 is measurable. They are also useful to decide which  phase portrait happens when we
 apply the Monte Carlo method to estimate the desired probabilities.

\begin{proof}[Proof of Theorem~\ref{t:teocub}] The phase portraits of cubic homogeneous
vector fields are given in~\cite{A,CL}. It can be seen that the only ones  with positive probability, modulus time orientation,
     are the 9 phase portraits given in Figure~\ref{f:dibuix3}. Set $c_j=P(C_j),$ $j=1,2,\ldots,9.$
     By looking at these phase portraits and with the notation introduced in Theorem~\ref{te:tot} it holds that
     \begin{align*}
     &u_3(-3)= c_1,\,\, u_3(-1)= c_2+c_6 ,\,\, u_3(1)= c_3+c_4+c_7+c_9,\,\, u_3(3)=c_5+c_8,\\
      &\ell_3(0)=c_9,\,\, \ell_3(2)=\sum_{j=6}^8 c_j,\,\, \ell_3(4)=\sum_{j=1}^5 c_j.
     \end{align*}
     By Theorem~\ref{te:tot} we know that  $u_3(3)=u_3(-3),$ $u_3(1)=u_3(-1)$  and
     $2\ell_3(2)+4\ell_3(4)=\Lambda_3.$ Joining all these equalities, the ones stated in the theorem follow.

Finally, notice that, contrary to what happen in the quadratic case,
the index at the origin and the number of invariant straight lines
are not enough to distinguish between different phase portraits.
 The phase portraits $C_3$ and $C_4$ have both index
equal to $1$ and have $4$ invariant straight lines.
\end{proof}

 If we look at the equator of the Poincar\'{e}
sphere we see that in the phase portrait $C_4$ the infinite singular
points are successively node-saddle-node-saddle. It can be seen (see~\cite{A}) that
if $t_3(\kappa_j)=0$, then $s_j:=-t_3'(\kappa_j)p_3(1,\kappa_j)<0$
(resp. $s_j>0$) is the condition to get a saddle (resp. a node) at
the singular point at infinity determined by the direction
$y=\kappa_j x.$  Hence, in the case with index $1$ and
4 invariant straight lines, in order to distinguish between the  phase portraits $C_3$ and $C_4$ it is
necessary to consider the four roots of the polynomial equation $t_3(\kappa)=0,$
 $\kappa_1<\kappa_2<\kappa_3<\kappa_4$, and compute
the four signs $s_j$. If  we find that $s_1s_2<0$ and $s_2s_3<0$,
then we have the phase portrait $C_4$. Otherwise we have phase portrait $C_3.$  We note that this type of
conditions on the coefficients  give measurable sets.

Hence in this case to get the values of Table~\ref{t:taula3} we use again the Monte Carlo method
generating $10^8$ random cubic homogeneous vector fields with the desired distribution. For each sample we
compute its index at the origin by using Theorem~\ref{indexcubics} and  we apply Lemma~\ref{l:app18} to
know its number of invariant straight lines. These two values are enough to know the corresponding
phase portrait in seven of the nine cases.  To distinguish
between the phase portraits $C_3$ and $C_4$, in the case  $(i,l)=(1,4),$  we compute the values introduced in the previous paragraph.

\section*{Appendix}
 In this appendix we give
conditions on the coefficients of polynomial equations of degree~$3$
or~$4$ to know its number real zeros zeros
and for degree $3$ also their signs.

\begin{lem}\label{signaturacubics}
Set $p(x)=x^3+ax^2+bx+c,$  $c_2=ab-9c,$ $d_2=a^2-3b$ and
$d_3=-27c^2+18abc-4a^3c+a^2b^2-4b^3.$ Assume that $a\cdot b\cdot c
\cdot c_2\cdot d_2\cdot  d_3\ne0.$ Then the following holds:
\begin{itemize}

\item [(i)] $p$ has a unique real root of multiplicity one, if and only if $d_3<0.$
This root is negative (resp. positive) if  $c>0$ (resp. $c<0$).

\item [(ii)] $p$ has three negative real roots if and only if
$d_3>0,\,c_2>0,\,b>0,c>0.$

\item [(iii)] $p$ has two negative roots and one positive one if and only if
either, $d_3>0,\,c<0,c_2>0$ or $d_3>0,\,c<0,\,c_2<0,b<0.$

\item  [(iv)] $p$ has one negative root and two positive roots if and only if
either, $d_3>0,\,c>0,c_2<0$ or $d_3>0,\,c>0,\,c_2>0,b<0.$

\item [(v)] $p$ has three positive roots if and only if
$d_3>0,\,c_2<0,\,b>0,c<0.$
\end{itemize}
\end{lem}
\begin{proof}
The proof is based in Sturm's method which asserts that if
$\big(p_0=p,p_1\ldots,p_n\big)$ is a Sturm's sequence of $p$ in $[a,b]$
with $p(a)\cdot p(b)\ne 0,$ then the number of real zeros of $p$ in
$(a,b)$ is $V(a)-V(b)$ where $V(x)$ is the number of changes of sign
in the ordered sequence $\big(p_0(x),p_1(x),\ldots,p_n(x)\big),$ where the zeroes are disregarded.
For any polynomial without multiple roots such a sequence always exists, see~\cite{SB}. For
$p,$ without multiple roots and $d_2\ne0,$ one Sturm sequence is $p_0=p,$ $p_1=p'$ and
\[
p_2(x) = \frac{1}{9}\left(2d_2x+c_2\right),\quad
p_3(x) = \frac{9}{4\,d_2^2}\,d_3.
\]
The quantity $d_3$ is the classical
discriminant of $p$ and it is known that $d_3<0$ if and only if $p$
has a real root and two complex ones whereas $d_3>0$ if and only if
$p$ has three real roots (see \cite{CL} for instance). We are going
to consider the following two tables depending on the sign of $d_3$ where $b,c,d_2,c_2$ stands for
their respective signs.

\smallskip

\begin{center}\begin{tabular}{cc}
        \begin{tabular}{|c|c c c|}
            \hline
            & $-\infty$ & $0$ & $\infty$  \\ [0.5ex]
            \hline\hline
            $p_0$& $-$ & $c$ & $+$  \\
            $p_1$&  $+$ & $b$ & $+$ \\
            $p_2$&  $-d_2$ & $c_2$ & $d_2$ \\
            $p_3$&  $-$ & $-$ & $-$\\ [1ex]
            \hline
        \end{tabular}\quad\quad\quad\quad\quad\quad
        &
        \begin{tabular}{|c|c c c|}
            \hline
            & $-\infty$ & $0$ & $\infty$  \\ [0.5ex]
            \hline\hline
            $p_0$& $-$ & $c$ & $+$  \\
            $p_1$&  $+$ & $b$ & $+$ \\
            $p_2$&  $-d_2$ & $c_2$ & $d_2$ \\
            $p_3$&  $+$ & $+$ & $+$\\ [1ex]
            \hline
        \end{tabular}\\
        $d_3<0$\quad\quad\quad\quad\quad\quad& $d_3>0$
\end{tabular}
\end{center}

\noindent We separate the proof in two cases, depending on the sign of $d_3.$

\textbf{Case 1.} Assume that $d_3<0.$
We observe that $V(-\infty)=2$ and $V(+\infty)=1.$ Then
\begin{itemize}
\item The polynomial $p$ has a negative real root if and only if
$V(0)=1.$ It is easy to see that it happens when $c>0$ and one of the three
following conditions hold:
\[c_2>0,b>0;\quad
c_2<0,b>0;\quad
c_2<0,b<0.\]
We observe that $d_3<0\,,\,c_2>0\,,\,b<0\,,\,c>0$ is not compatible
because then $V(0)-V(+\infty)=3$ but $V(-\infty)-V(+\infty)=1.$
Summarizing, $p$ has a negative real root and two more complex ones
if and only if $d_3<0,\,c>0.$

\item The polynomial $p$ has a positive real root if and only if
$V(0)=2.$ It is easy to see that it happens when $c<0$ and one of the three
following conditions hold:
\[
c_2>0,b>0;\quad
c_2>0,b<0;\quad
c_2<0,b>0.
\]
As before conditions $d_3<0,c_2<0,b<0,c<0$ are incompatible and then
$p$ has a positive real root and two more complex ones if and only
if $d_3<0,\,c<0$ as announced.
\end{itemize}

\noindent \textbf{Case 2.} Assume that $d_3>0$ and consider the above right-hand side table of signs.
We see that in that case $d_2$ must be positive. Otherwise
$V(-\infty)=1,\,V(\infty)=2$ which is not possible (we can also
argue that since $d_3>0$ implies that $p$ has three real roots, its
derivative has two real roots and hence its discriminant which is equal to
$4\,d_2$ has to be positive).

It is straightforward to see that items $(ii),(iii),(iv),(v)$ are
equivalent to $V(0)=0,V(0)=1,V(0)=2,V(0)=3$ respectively and that
these number of changes of sign are satisfied exactly when the
conditions on $b,c,c_2$ are the ones stated in the lemma.
~\end{proof}

\begin{lem}\label{l:app18}
Let $p(x)=x^4+ax^3+bx^2+cx+d.$ Assume that $c\cdot d\cdot c_2\cdot
c_3\cdot d_2\cdot d_3\cdot d_4\ne0.$ Then the following holds:
\begin{itemize}

\item [(a)] $p$ has two simple real and two complex roots if and
only if $d_4<0.$

\item [(b)] $p$ has four different real roots if and only if $d_4>0,d_2>0$
 and  $d_3>0.$

\item [(c)] $p$ has no real root  if and only if $d_4>0$
and either, $d_2<0$ or $d_3<0.$

\end{itemize}
\end{lem}

\begin{proof}
A  Sturm sequence of $p,$ is
$p_0=p,$  $p_1=p',$
\[
p_2(x) = \frac{1}{16}\left(d_2x^2+2(ab-2c)x+c_2\right),\,\,
p_3(x) = \frac{16(2d_3x+c_3)}{d_2^2},\,\,
p_4(x) = \frac{d_2^2d_4}{64d_3^2}.
\]
where
$$\begin{array}{l} d_2=3\,{a}^{2}-8\,b,\\c_2=ac-16 d,\\
d_3=-3\,{a}^{3}c+{a}^{2}{b}^{2}-6\,{a}^{2}d+14\,abc-4\,{b}^{3}+16\,bd-18\,
{c}^{2}
,\\c_3=-9\,{a}^{3}d+{a}^{2}bc+32\,abd+3\,a{c}^{2}-4\,{b}^{2}c-48\,cd,\\d_4=-27\,{a}^{4}{d}^{2}+18\,{a}^{3}bcd
-4\,{a}^{3}{c}^{3}-4\,
{a}^{2}{b}^{3}d+{a}^{2}{b}^{2}{c}^{2}+144\,{a}^{2}b{d}^{2}-6\,{a}^{2}{c}^{2}d-80\,a{
    b}^{2}cd+\\\qquad+18\,ab{c}^{3}+16\,{b}^{4}d-4\,{b}^{3}{c}^{2}-192\,ac{d}^{2}-
128\,{b}^{2}{d}^{2}+144\,b{c}^{2}d-27\,{c}^{4}+256\,{d}^{3}.
\end{array}$$
We note that $d_4$ is the discriminant of $p$ and since
$d_4\ne 0$ all its roots are simple.

It is known, see \cite{CL} for instance, that $d_4<0$ if and only if
$p$ has two simple real and two complex roots. And that $d_4>0$ if
and only if $p$ has either, four different real roots or  no real root. It is very easy to distinguish between these
last two possibilities using the Sturm method. Consider the
corresponding table when $d_4>0,$ where again each value stands for its sign.
\begin{table}[H]
\centering
\begin{tabular}{|c|c c c|}
 \hline
 &$-\infty$ & $0$ & $\infty$  \\ [0.5ex]
 \hline\hline
$p_0$&  $+$ & $d$ & $+$  \\
$p_1$& $-$ & $c$ & $+$  \\
$p_2$&  $d_2$ & $c_2$ & $d_2$ \\
$p_3$&  $-d_3$ & $c_3$ & $d_3$ \\
$p_4$&  $+$ & $+$ & $+$\\ [1ex]
 \hline
 \end{tabular}
\end{table}

If $p$ has four real roots then $V(-\infty)-V(\infty)$ must be four
and this only can happen if $V(-\infty)=4$ and $V(\infty)=0$ what
immediately says that $d_2$ and $d_3$ must be positive. If $d_4>0$
and $d_2<0$ or $d_3<0$ then $V(-\infty)-V(+\infty)=2-2=0$ and the
result follows.~\end{proof}

\section*{Acknowledgments}    We want to acknowledge Prof. Maria Jolis for her helpful indications.

\medskip

\noindent We also acknowledge the anonymous reviewer for the careful reading and for addressing us very interesting comments. Some of his/her indications are summarized in Remark \ref{re:invariants}.  We also thank him/her for sharing with us the results of his/her Monte Carlo  simulations using the approximation of that remark, which are consistent with our results.

\section*{Data availability}    All data analyzed in this study are included in this article.

\end{document}